\documentclass[10pt]{article}

\usepackage{graphicx}

\usepackage{amssymb}

\usepackage{amsthm}

\usepackage{epstopdf}

\usepackage{amsmath}
\usepackage{amsfonts}

\def\phi{{\varphi}}

\DeclareSymbolFont{AMSb}{U}{msb}{m}{n}
\DeclareMathSymbol{\N}{\mathbin}{AMSb}{"4E}
\DeclareMathSymbol{\Z}{\mathbin}{AMSb}{"5A}
\DeclareMathSymbol{\R}{\mathbin}{AMSb}{"52}
\DeclareMathSymbol{\Q}{\mathbin}{AMSb}{"51}
\DeclareMathSymbol{\I}{\mathbin}{AMSb}{"49}
\DeclareMathSymbol{\C}{\mathbin}{AMSb}{"43}

\def\be{\begin{equation}}
\def\ee{\end{equation}}
\def\ber{\begin{eqnarray}}
\def\eer{\end{eqnarray}}

\def\beq{\begin{equation}}
\def\eeq{\end{equation}}

\def\Z{{\mathbb{Z}}}

\begin{document}

\addtolength{\textheight}{0 cm} \addtolength{\hoffset}{0 cm}
\addtolength{\textwidth}{0 cm} \addtolength{\voffset}{0 cm}

\newenvironment{acknowledgement}{\noindent\textbf{Acknowledgement.}\em}{}

\setcounter{secnumdepth}{5}

 \newtheorem{proposition}{Proposition}[section]
\newtheorem{theorem}{Theorem}[section]
\newtheorem{lemma}[theorem]{Lemma}
\newtheorem{coro}[theorem]{Corollary}
\newtheorem{remark}[theorem]{Remark}
\newtheorem{extt}[theorem]{Example}
\newtheorem{claim}[theorem]{Claim}
\newtheorem{conj}[theorem]{Conjecture}
\newtheorem{definition}[theorem]{Definition}
\newtheorem{application}{Application}

\newtheorem*{thm*}{Theorem A}

\newtheorem{corollary}[theorem]{Corollary}

\title{A variational principle for   problems  with a hint of    convexity \\
\footnote{The author is pleased to acknowledge the support of the  National Sciences and Engineering Research Council of Canada.
}}
\author{
Abbas Moameni
\hspace{2mm}\\
{\it\small School of Mathematics and Statistics}\\
{\it\small Carleton University,}
{\it\small Ottawa, ON, Canada  }\\
{\it\small  momeni@math.carleton.ca}\\\\
}

\date{}

\maketitle

\vspace{3mm}

\section*{Abstract}
A variational principle is introduced to provide a new formulation and resolution for several boundary value problems with a variational structure. This principle allows one to deal with problems well beyond the weakly compact structure. As a result, we study several super-critical semilinear  Elliptic problems.

\noindent
{\it \footnotesize 2010 Mathematics Subject Classification:    35J15, 	58E30}. {\scriptsize }  	 \\
{\it \footnotesize Key words:      Variational principles, supercritical Elliptic problems}. {\scriptsize }

\section{Introduction}

Let $V$ be a  real Banach  space and $V^*$ its topological dual  and let $\langle .,. \rangle $ be the pairing between $V$ and $V^*.$ Let $\Psi : V \rightarrow \mathbb{R}\cup \{+\infty\}$ be a proper convex and lower semi continuous  function and let $K$ be a convex and weakly closed subset of  $ V.$  Assume that $\Psi$ is 
G\^ateaux differentiable on $K$ and denote by $D \Psi$ the G\^ateaux derivative of $\Psi.$ 
Let  $\Phi\in C^1(V, \R)$  and consider the following problem,  
\begin{equation}\label{p1}
\text{Find } u_0 \in K \text{ such that } D \Psi(u_0)=D \Phi(u_0).
\end{equation}
 The restriction of $\Psi$ to $K$ is denoted by $\Psi_K$  and defined by 
\begin{eqnarray*}\Psi_K(u)=\left\{ \begin{array}{cc}
\Psi(u),&  \quad u \in K,\\
+\infty, & \quad u \not \in K.
\end{array} 
\right.
\end{eqnarray*}
To find a solution for (\ref{p1}), we shall consider the critical points of the functional $I: V \rightarrow \mathbb{R}\cup \{+\infty\}$ defined  by 
\[I(u):=\Psi_K(u)-\Phi(u).\]

According to  Szulkin \cite{Su}  we have the following definition for critical points  of $I$ (see also the appendix).
\begin{definition}\label{p3} A point $u\in V$ is said to be a critical point of $I$ if $I(u)\in \R$ and  it satisfies  the inequality 
\begin{eqnarray}\label{p2}
\Psi_K(v)-\Psi_K(u) \geq \langle D \Phi(u), v-u\rangle, \qquad \forall v \in V.
\end{eqnarray}
\end{definition}
Note that a function $u$  satisfying (\ref{p2}) is indeed a solution of  the inclusion $D \Phi(u) \in \partial \Psi_K(u).$ Therefore, it is not necessarily a solution of (\ref{p1}) unless $D=V.$ There is a  well developed theory to find critical points of functionals of the form $I$. We refer the interested reader to \cite{Su, M-P}. Here is our main result in this paper.

\begin{theorem}[Variational Principle]\label{main}
Let $\Psi : V \rightarrow \mathbb{R}\cup \{+\infty\}$ be a proper convex and lower semi continuous  function and let $K$ be a convex and weakly closed subset of  $ V.$ Assume that $\Psi$ is 
G\^ateaux differentiable on $K$ and $\Phi\in C^1(V, \R)$. 
  If the following two assertions hold:
 \begin{enumerate}
 \item[(i)]The functional $I: V \rightarrow \mathbb{R}\cup \{+\infty\}$ defined  by 
$I(u)=\Psi_K(u)-\Phi(u)$ has a critical point $u_0\in V,$ and;
 \item[(ii)] there exists $v_0 \in K$ such that $D \Psi(v_0)=D\Phi(u_0).$
 \end{enumerate}
 Then $u_0\in K$ is a solution of (\ref{p1}), that is, 
  \[D \Psi(u_0)=D\Phi(u_0).\]
\end{theorem}

The above theorem has many interesting applications in partial differential equations . We shall briefly recall some of them and refer the interested reader to \cite{Moo} where some more general  versions of Theorem \ref{main} are established and several applications in the fixed point theory and PDEs are provided. It is also worth noting that Theorem \ref{main} extends some of variational principles established by the author in \cite{Mo1, Mo2}.\\
We shall now proceed with some applications.

\subsection{A concave-convex nonlinearity}
We consider the problem
\begin{equation}\label{con-c}
\left \{
\begin{array}{ll}
-\Delta u =|u|^{p-2} u+\mu |u|^{q-2}u, &   x \in \Omega\\
u=0, &  x \in \partial \Omega
\end{array}
\right.
\end{equation}
where $\Omega\subset \R^n$ is a bounded domain with $C^1$-boundary and  $1<q\leq 2<p.$ This problem was studied by Ambrosetti and etc. in \cite{A-B-C} and Bartsch and Willem in \cite{B-W}. Our plan is to show that for positive $\mu$ and $p$ bigger that the critical exponent $2^*=2n/(n-2),$ problem (\ref{con-c}) has a strong  solution  in $H^2(\Omega).$ \\
Let $V=H^2(\Omega)\cap H_0^1(\Omega),$ and let  $I:V \to \R $ be the Euler-Lagrange functional corresponding to (\ref{non}),  
\[I(u)=\frac{1}{2}\int_\Omega |\nabla u|^2 \, dx -\frac{1}{p}\int_\Omega | u|^p \, dx-\frac{1}{q}\int_\Omega |u|^q\, dx.\]
For $r>0$, define  the convex set $K(r)$ by
\[K(r)=\Big \{u \in H^2(\Omega)\cap H_0^1(\Omega); \, \|u\|_{H^2(\Omega)}\leq r\Big \}.\]
We have the following result
\begin{theorem}\label{C-C} Assume that $1< q< 2< p <p^*$ where $p^*=(2n-4) /(n-4)$ for $n > 4$ and $p^*=\infty $ for $n\leq 4.$ 
Then there exists $\mu^*>0$ such that for each $\mu  \in(0, \mu^*)$  problem (\ref{con-c}) has a  non-trivial solution. Indeed, for each $\mu  \in(0, \mu^*),$ there exist positive  numbers  $r_1, r_2 \in \R$ with $r_1 <r_2$ such that  for each $r \in [r_1,r_2]$ the problem (\ref{con-c}) has a  solution $u \in K(r)$ with $I(u)<0.$ 
\end{theorem}
\begin{proof} 
We apply Theorem \ref{main}, where 
\[\Psi(u)=\frac{1}{2}\int_\Omega |\nabla u|^2\, dx, \qquad \Phi(u)=\frac{1}{p}\int_\Omega | u|^p \, dx+\int_\Omega |u|^q\, dx,\]
and 
$K:=K(r)$
for some  $r>0$ to be determined. Note that the Sobolev space  $H^2(\Omega)$ is compactly embedded in $L^t(\Omega)$ for 
$t < t^*$ where $t^*=2n/(n-4)$ for $n>4,$ and $t^*=+\infty$ for $n \leq 4.$   It then  follows that the function $\Phi$  is continuously differentiable for $p<p^*.$ By standard methods, there exists $u_0\in K(r)$ such that 
\[I(u_0)=\min_{u\in K(r)} I(u).\]
Since $1<q<2<p$ and $\mu>0,$
it is easily seen that $I(u_0)<0$ and therefore $u_0\not\equiv 0$ is a critical point of $I$ restricted to $K(r).$
To verify  condition $(ii)$ in Theorem \ref{main}, we show that there exists $v_0\in K(r)$ such that $-\Delta v_0=|u_0|^{p-2}u_0+\mu |u_0|^{q-2}u_0.$  The existence of such $v_0$ follows by standard arguments.  We  show that $v_0 \in K(r)$ for $r$ small. It follows from the Elliptic regularity theory (see Theorem 8.12 in  \cite{G-T}) that 
\begin{eqnarray*}\|v_0\|_{H^2(\Omega)} & \leq & C\Big (\big \||u_0|^{p-2}u_0\big \|_{L^2(\Omega)}+\mu\big \||u_0|^{q-2}u_0\big \|_{L^2(\Omega)}\Big )\\
& = & C\Big (\big \|u_0\big \|^{p-1}_{L^{2(p-1)}(\Omega)}+\mu\big \|u_0\big \|^{q-1}_{L^{2(q-1)}(\Omega)}\Big ),
\end{eqnarray*}
where $C$ is a constant depending on $\Omega.$
Since $2(q-1)<2(p-1)<t^*,$ we obtain that 
\begin{eqnarray*}\|v_0\|_{H^2(\Omega)} 
& \leq & C_1\Big (\big \|u_0\big \|^{p-1}_{H^2(\Omega)}+\mu\big \|u_0\big \|^{q-1}_{H^2(\Omega)}\Big )\\
& \leq & C_1(r^{p-1}+\mu r^{q-1}).
\end{eqnarray*}
where  $C_1$ is a   constant in terms of $p, q$ and $ \Omega.$ Choose $\mu^*>0$ small enough such that for each $\mu  \in(0, \mu^*),$ there exist positive  numbers  $r_1, r_2 \in \R$ with $r_1 <r_2$ such that  $C_1(r^{p-1}+\mu r^{q-1})\leq r$ for all $r \in [r_1, r_2].$ It then follows that $v_0 \in K(r)$ provided $\mu  \in(0, \mu^*)$ and  $r\in [r_1,r_2]. $

\end{proof}

\subsection{Non-homogeneous semilinear Elliptic equations}
Here we shall consider the problem
\begin{eqnarray}\label{non}\left\{ \begin{array}{ll}
-\Delta u=|u|^{p-2} u+ f(x),&  \quad  x\in \Omega,\\
u=0, & \quad x \in \partial \Omega.
\end{array} 
\right.
\end{eqnarray}
where $\Omega$ is on open bounded domain  in $\R^n$ with $C^1$-boundary.  Problem (\ref{non}) was treated in \cite{B-B, St} for $p$ less than the critical exponent $2^*.$  As an application of Theorem \ref{main} together with Elliptic regularity theory  we shall show that problem (\ref{non}) has a solution for $p$ beyond the critical Sobolev exponent. In this case, the standard variational methods fail to work.  Note that our approach can be  applied  to more general nonlinearities (see \cite{Moo}).    We have the following theorem.
\begin{theorem}
Let  $ 2< p <p^*$ where $p^*=(2n-4) /(n-4)$ for $n > 4$ and $p^*= \infty $ for $n\leq 4.$  There exists $\lambda>0$ such that for $\|f\|_{L^2(\Omega)}< \lambda,$  problem   (\ref{non}) has a solution $u\in H^2(\Omega).$ 
\end{theorem}

\begin{proof} Let $V=H^2(\Omega)\cap H_0^1(\Omega),$ and let  $I:V \to \R $ be the Euler-Lagrange functional corresponding to (\ref{non}),  
\[I(u)=\frac{1}{2}\int_\Omega |\nabla u|^2 \, dx -\frac{1}{p}\int_\Omega | u|^p \, dx-\int_\Omega f u\, dx.\]
We apply Theorem \ref{main}, where 
\[\Psi(u)=\frac{1}{2}\int_\Omega |\nabla u|^2\, dx, \qquad \Phi(u)=\frac{1}{p}\int_\Omega | u|^p \, dx+\int_\Omega f u\, dx,\]
and 
\[K:=K(r)=\big \{u \in H^2(\Omega)\cap H_0^1(\Omega); \, \|u\|_{H^2(\Omega)}\leq r\big \},\]
for some  $r>0$ to be determined.  By standard methods, there exists $u_0\in K(r)$ such that 
\[I(u_0)=\min_{u\in K(r)} I(u).\]
To verify  condition $(ii)$ in Theorem \ref{main}, one needs to show  that there exists $v_0\in K(r)$ such that $-\Delta v_0=|u_0|^{p-2}u_0+f(x).$ Existence of $v_0 \in H^2(\Omega)$ is standard. The fact that
$v_0\in K(r)$ for $\|f\|_{L^2(\Omega)}$ small, follows  by the Elliptic regularity theory and the argument made in the proof of Theorem \ref{C-C}.
\end{proof}

\subsection{Super critical Neumann problems}

We shall  consider the existence of positive solutions of the  Neumann problem
\begin{equation}\label{eq}
\left\{\begin{array}{ll}
-\Delta u + u= a(x)|u|^{p-2}u, &   x\in B_1 \\
u>0,    &   x\in B_1, \\
\frac{\partial u}{\partial \nu}= 0, &    x\in  \partial B_1,
\end {array}\right.
\end{equation}
where $B_1$ is the unit ball  centered at the origin  in $\mathbb{R}^N$, $N \geq 3,$ $p>2.$ and $a$ is a radial function, i.e., $a(x)=a(r)$ where $r=|x|.$ 
\begin{theorem}\label{exi}
Assume that  $a\in L^{\infty}(0,1)$ is increasing, not constant and $a(r)>0$ a.e. in $[0,1]$. Then problem (\ref{eq}) admits at least one radially increasing positive solution.
\end{theorem}
{\it Sketch of the proof.}
Let $V=L^p(\Omega)\cap H_r^1(\Omega),$ where $H^1_r$ is the set of radial functions in $H^1(\Omega).$ We apply Theorem \ref{main}, where 
\[\Psi(u)=\int_\Omega \frac{|\nabla u|^2+u^2}{2}\, dx, \qquad \Phi(u)=\frac{1}{p}\int_\Omega a(x) | u|^p \, dx,\]
and  
\[K= \big \{u\in V: u(r)\geq  0, u(r)\leq u(s), \forall r,s \in [0,1], r\leq s\big \}.\]
It can be easily deduced that  that $V \cap K$ is continuously embedded in $L^{\infty}(\Omega)$
from which one can apply Theorem \ref{MPT} to show that $I=\Psi-\Phi$ restricted to $K$ has a critical point $u_0\in K$ of mountain pass type (See \cite{C-M} for a detailed argument). 
It is also established in \cite{C-M} that there exists $v_0\in K $ satisfying  $-\Delta v_0 + v_0= a(|x|)|u_0|^{p-2}u_0.$  Thus,  by Theorem \ref{main}, $u_0$ is a non-negative and nontrivial solution of (\ref{eq}). It also follows from the maximum principle that $u_0$ is indeed positive.
\hfill$\square$\\

We remark that finding radially increasing solutions of  problems of type  (\ref{eq}) has been the subject of many studies in recent years  starting the works of \cite{B-N-W, G-N, S-T}. 
\section{Proof of the variatinal principle.}

In this section we shall prove Theorem \ref{main}. We first recall some important definitions and results from   convex analysis.

Let $V$ be a  real Banach  space and $V^*$ its topological dual  and let $\langle .,. \rangle $ be the pairing between $V$ and $V^*.$
Let $\Psi : V \rightarrow \mathbb{R}\cup \{\infty\}$ be a proper convex  function. The subdifferential $\partial \Psi $ of $\Psi$
is defined  to be the following set-valued operator: if $u \in Dom (\Psi)=\{v \in V; \, \Psi(v)< \infty\},$ set
\[\partial \Psi (u)=\big \{u^* \in V^*; \langle u^*, v-u \rangle + \Psi(u) \leq \Psi(v) \text{  for all  } v \in V\big \}\]
and if $u \not \in Dom (\Psi),$ set $\partial \Psi (u)=\varnothing.$ If $\Psi$ is G\^ateaux differentiable at $u,$ denote by $D \Psi(u)$ the derivative of $\Psi$ at $u.$ In this case  $\partial \Psi (u)=\{ D  \Psi(u)\}.$\\
 The  Fenchel  dual of an arbitrary function $\Psi$ is denoted by  $\Psi^*,$ that is function on $V^*$ and is defined by
\[\Psi^*(u^*)=\sup \{\langle u^*, u \rangle- \Psi (u); u \in V\}.\]
Clearly $\Psi^*: V^* \rightarrow \mathbb{R}\cup \{+\infty\}$  is convex and  weakly lower semi-continuous.  The following standard  result  is crucial in the subsequent analysis (see \cite{Ek-Te} for a proof).
\begin{proposition}\label{var-pro} Let $\Psi : V \rightarrow \mathbb{R}\cup \{+\infty\}$ be convex and lower-semi continuous. 
 then $\Psi^{**}=\Psi$ and the following holds:
\[
  \Psi (u) +\Psi^*(u^*) = \langle u, u^* \rangle \quad   \Longleftrightarrow \quad 
 u^* \in \partial \Psi (u)\quad  \Longleftrightarrow \quad 
  u \in \partial \Psi^* (u^*).\]
\end{proposition}
\textbf{Proof of Theorem \ref{main}.} Since $u_0$ is a critical point of $I(u)=\Psi_K(u)-\Phi(u),$
it follows from Definition \ref{p3} that 
\begin{equation}\label{ineq0}
\Psi_K(v)-\Psi_K(u_0)\geq \langle D \Phi(u_0),v-u_0\rangle,\quad \forall v\in V.
\end{equation}
It follows from $(i)$ and $(ii)$ in the theorem that $u_0, v_0 \in K$ and $D \Psi(v_0)=D\Phi(u_0)$. Thus, it follows from  inequality  (\ref{ineq0}) with  $v=v_0$ that 
\begin{equation}\label{ineq2}
\Psi(v_0)-\Psi(u_0)\geq \langle D \Psi(v_0),v_0-u_0\rangle.
\end{equation}
 Since $\Psi$ is G\^ateaux differentiable at $v_0\in K,$ it follows that $\partial \Psi(v_0)=\{D \Psi(v_0)\}$  which together with 
 the convexity of $\Psi$ one obtains  that
\begin{equation}\label{ineq1}
\Psi(u_0)-\Psi(v_0)\geq \langle D \Psi(v_0),u_0-v_0\rangle.
\end{equation}
It follows from (\ref{ineq2}) and (\ref{ineq1}) that
\begin{equation}\label{ineq3}
\Psi(v_0)-\Psi(u_0)=\langle D \Psi(v_0),v_0-u_0\rangle.
\end{equation}
We now claim that  $D \Psi(v_0)=D \Psi(u_0)$ from which the desired result follows,
\[D \Psi(u_0)=D \Psi(v_0)=D\Phi(u_0).\]
\,\,\,\,{\it Proof of the claim:} 
Let $w^*=D \Psi(v_0).$ Since $\Psi$ is convex and lower semi continuous it follows from Proposition \ref{var-pro} that 
 \begin{equation}\label{eq2}\Psi(v_0)+ \Psi^*(w^*)=\langle w^*, v_0\rangle.
 \end{equation}
It now follows from (\ref{ineq3}) and (\ref{eq2}) that 
 \begin{eqnarray*}
 \langle w^*, u_0\rangle-\Psi(u_0)=\langle w^*, v_0\rangle-\Psi(v_0)=\Psi^*(w^*),
 \end{eqnarray*}
 from which one obtains
 \[\Psi(u_0)+ \Psi^*(w^*)=\langle w^*, u_0\rangle.\]
 This indeed implies that $w^* \in \partial \Psi(u_0)$ by virtue of Proposition \ref{var-pro}. Since $\Psi$ is  G\^ateaux differentiable at $u_0$ we have that $\partial \Psi(u_0)=\{D \Psi(u_0)\}$. Therefore, 
 \[D \Psi(u_0)=w^*=D \Psi(v_0),\]
 as claimed.
\hfill $\square$

\section{Appendix} \label{convex}

We shall now recall some notations and results for the minimax principles of  lower semi-continuous functions used throughout the paper.  
\begin{definition}
Let $V$ be a real Banach space,  $\Phi\in C^1(V,\mathbb{R})$ and $\Psi: V\rightarrow (-\infty, +\infty]$ be proper (i.e. $Dom(\Psi)\neq \emptyset$), convex and lower semi-continuous.
A point $u\in  V$ is said to be a critical point of \begin{equation} \label{form}I:=  \Psi-\Phi \end{equation} if $u\in Dom(\Psi)$  and if it satisfies
the inequality
\begin{equation}
 <D \Phi(u), u-v>+ \Psi(v)- \Psi(u) \geq 0, \qquad \forall v\in V.
\end{equation}
\end{definition}

\begin{definition}
We say that $I$ satisfies the Palais-Smale compactness  condition (PS)   if
every sequence $\{u_n\}$ such that $I(u_n)\rightarrow c\in \mathbb{R},$ and 
\[  <D \Phi(u_n), u_n-v>+ \Psi(v)- \Psi(u_n) \geq -\epsilon_n\|v- u_n\|, \qquad \forall v\in V,
\]
where $\epsilon_n \rightarrow 0$, then $\{u_n\}$ possesses a convergent subsequence.
\end{definition}

The following is  proved in \cite{Su}.
\begin{theorem}\label {MPT}
(Mountain Pass Theorem).  Suppose that
$I : V \rightarrow (-\infty, +\infty ]$ is of the form (\ref{form}) and satisfies the Palais-Smale   condition and  the Mountain Pass Geometry (MPG):
\begin{enumerate}
\item $I(0)= 0$.
and  there exists $e\in V$ such that $I(e)\leq 0$.
\item there exists some $\rho$ such that $0<\rho<\|e\|$ and for every $u\in V$ with $\|u\|= \rho$ one has $I(u)>0$.
\end{enumerate}
Then $I$ has a critical value $c\geq \rho$ which is  characterized by
$$c= \inf_{g\in \Gamma}  \sup_{t\in [0,1]} I[g(t)],$$
where   $\Gamma= \{g\in C([0,1],V): g(0)=0, g(1)= e\}.$
\end{theorem}

\end{document}